\documentclass[reqno]{amsart}
\usepackage[colorlinks=true, pdfstartview=FitV, linkcolor=blue, citecolor=blue]{hyperref}

\usepackage{amssymb,amsmath,amsfonts}
\usepackage{bm}
\usepackage{enumitem}
\usepackage{array}

\usepackage{graphicx}

\allowdisplaybreaks

\newcommand{\N}{\mathbb{N}}
\newcommand{\R}{\mathbb{R}}

\newcommand{\Z}{\mathbb{Z}}

\newcommand{\rar}{\rightarrow}
\newcommand{\mc}{\mathcal}

\newtheorem*{theorem1*}{Theorem \cite[Theorem 1.1]{Afif2024}}
\newtheorem{theorem}{Theorem}
\newtheorem{lemma}[theorem]{Lemma}

\allowdisplaybreaks

\begin{document}
	
	\title[Matrices with fixed determinant and bounded coefficients]{Counting $2\times 2$ matrices with fixed determinant and bounded coefficients}
	
	\author{Kavita Dhanda, Alan Haynes, Silmi Prasala}

	\thanks{Research supported by a Summer Undergraduate Research Fellowship (SURF) from the University of Houston.\\
		\phantom{A..}MSC 2020: 11D04; 11D45; 11N45}
	

	\begin{abstract}
	Recent work by M. Afifurrahman established the first asymptotic estimates with error terms for the number of $2\times 2$ matrices with fixed non-zero determinant $n\in\mathbb{N}$, and with coefficients bounded in absolute value by $X$. In this paper we present a new proof of this result, which also gives an improved error term as $X\rightarrow\infty$. Similar to Afifurrahman's result, our error term is uniform in both $n$ and $X$, and our estimates are significant for $X$ as small as $n^{1/2+\delta}$. To complement this, we also demonstrate that the exponent $1/2+\delta$ in this statement cannot be reduced, by establishing a result which gives a different asymptotic main term when $n$ is either a prime or the square of a prime, and when $X=n^{1/2}$.
	\end{abstract}
	
	\maketitle

\section{Introduction}\label{sec.Intro}
It is a well studied problem to estimate the number $S_{k,\|\cdot\|}(n,X)$ of $k\times k$ matrices $A$ with integer coefficients, $\det(A)=n$, and $\|A\|\le X$, where $k\ge 2$, $n\in\Z$, $X\ge 1$, and $\|\cdot\|$ is a prescribed norm. For the Euclidean norm $\|\cdot\|_2$, a famous result of Selberg from the 1970's on the hyperbolic circle problem (see p.5 of Document 3 in Selberg's unpublished notes \cite{Selb1970}) gives that
\begin{equation*}
	S_{2,\|\cdot\|_2}(1,X)=6X^2+O_\epsilon\left(X^{4/3+\epsilon}\right),
\end{equation*}
for any $\epsilon>0$.

In 1993 Duke, Rudnick, and Sarnak published a landmark paper on the density of integer points on affine homogeneous varieties \cite{DukeRudnSarn1993}. As one application of their results they proved (Theorem 1.10 in their paper) that, for any $k\ge 2$ and for any $n\not= 0$,
\begin{equation}\label{eqn.DRS0}
	S_{k,\|\cdot\|_2}(n,X)=c_{k,n}X^{k^2-k}+O_{n,\epsilon}\left(X^{k^2-k-1/(k+1)+\epsilon}\right),
\end{equation}
for any $\epsilon>0$, where
\begin{equation*}
c_{k,n}=\frac{\pi^{k^2/2}n^{-(k-1)}}{\Gamma\left(\frac{k^2-k+2}{2}\right)\Gamma\left(\frac{k}{2}\right)\zeta(2)\cdots\zeta(k)}\prod_{p^\alpha\| n}\frac{(p^{\alpha+1}-1)\cdots (p^{\alpha+k-1}-1)}{(p-1)\cdots (p^{k-1}-1)}.
\end{equation*}
Analogous results for the case when $n=0$ were established by Katznelson in \cite{Katz1993} (cf. \cite{Katz1994}). These types of problems were further investigated, using tools from ergodic theory, by Eskin, Mozes, and Shah \cite{EskiMozeShah1996,Shah2000}, who proved asymptotic formulas for numbers of matrices with bounded Euclidean norm and specified characteristic polynomial.

Surprisingly, the asymptotic behavior of $S_{k,\|\cdot\|_\infty}(n,X)$, where $\|\cdot\|_\infty$ is the sup-norm, has been less well studied, even in the case when $k=2$, until recently. The above mentioned work of Selberg, Duke, Rudnick, and Sarnak relies on the rotational invariance of the Euclidean norm, and therefore does not apply to this problem. Asymptotic formulas for $S_{2,\|\cdot\|_\infty}(0,X)$ (with second order asymptotics) can be derived from work of Ayyad, Cochrane, and Zhang \cite{AyyaCochZhan1996}. There is an asymptotic formula for $S_{2,\|\cdot\|_\infty}(1,X)$ due to Roettger \cite{Roet2005} (who also considers an analogous problem over rings of integers of algebraic number fields), as well as more recent work of Bulinski and Shparlinski \cite{BuliShpa2024}, which leads to an improved error term (see equation (1.4) in \cite{Afif2024}). A `smoothed' version of the sup-norm problem for $k=2$ and arbitrary $n$ (actually for $2\times 2$ matrices with a fixed characteristic polynomial) is considered in \cite{Guri2024}, however adding smooth weights in the count does change the problem.

Significant progress on this problem was made by M.~Afifurrahman, who in July 2024 published a preprint containing a proof  of the following result.
\begin{theorem1*} For $n\in\N$,
	\begin{equation}\label{thm.Afif}
	S_{2,\|\cdot\|_\infty}(n,X)=\frac{96\sigma_{-1}(n)}{\pi^2}X^2+O\left(X^{o(1)}\max\left\{X^{5/3},n\right\}\right),
\end{equation}
as $X,n\rar\infty$.
\end{theorem1*}
In the statement of the theorem, $\sigma_s(n)$ denotes the sum of the $s^{\text{th}}$ powers of the positive divisors of $n$. One important feature of this result is that, in contrast to \eqref{eqn.DRS0}, the dependence on $n$ of the error term here is made explicit, so that both $X$ and $n$ may tend to infinity as functions of one another. This is necessary for some applications (cf. \cite[Theorem 4]{Shpa2010}, \cite[Section 2.2]{OstaShpa2025}), and the theorem is significant for $X$ even as small as $n^{1/2+\delta}$. For the case of large $X$, a preprint of Ganguly and Guria \cite{GangGuri2024} from October 2024  improves the error term in this theorem to $\ll_\epsilon n^\theta X^{3/2+\epsilon}$ (where $\theta$ is a non-negative constant), in the range where $X\gg n^3$.

Our first goal in this paper is to give a short proof of the following theorem.
\begin{theorem}\label{thm.Main0}
	Let $\epsilon>0$. Then, for $n\in\N$ and $X\ge 1$ we have that
	\begin{equation*}
		S_{2,\|\cdot\|_\infty}(n,X)=\frac{96\sigma_{-1}(n)}{\pi^2}X^2+O_\epsilon\left(\sigma_0(n)X\log X+nX^\epsilon\right).
	\end{equation*}
\end{theorem}
The bound in our theorem recovers the error term from \eqref{thm.Afif} when $X$ is small compared to $n$, but it gives a significant power savings (over both that result and the one in \cite{GangGuri2024}) when $X\gg n$. As in Afifurraman's result, our error term allows one to obtain significant estimates for $X$ as small as $n^{1/2+\delta}$. It turns out that this is the threshold of what one could hope for in the exponent, as demonstrated by our second result.
\begin{theorem}\label{thm.Main1}
	For all prime numbers $p$ we have that
	\begin{equation*}
		S_{2,\|\cdot\|_\infty}(p,p^{1/2})=4\left(\frac{12}{\pi^2}-1\right)p+O\left(p^{1/2}\log p\right),
	\end{equation*}
	and that
	\begin{equation*}
	S_{2,\|\cdot\|_\infty}(p^2,p)=4\left(\frac{12}{\pi^2}-1\right)p^2+O\left(p\log p\right).
\end{equation*}
\end{theorem}
Note that the constant on the main term in these formulas is different from that in Theorem \ref{thm.Main0}. This shows, in particular, that the main term in Theorem \ref{thm.Main0} does not give a correct asymptotic estimate, as $n\rar\infty,$ with  $X=n^{1/2}$.

Results similar to Theorem \ref{thm.Main1} for higher powers of primes do not seem to follow from the same types of arguments given in our proof. This is a direction for further investigation, which was the initial motivation for the research leading to the results in this paper, and which is closely related to a natural analogue of the `minimal denominator problem' in the setting of $p$-adic numbers (see \cite{ChenHayn2023,Mark2024,Shpa2024}).


\noindent \textbf{Notation:} We use $O$ and $\ll$ to denote the standard big-oh and Vinogradov notations. If $a$ and $q$ are integers then we write $(a,q)$ for the $\mathrm{gcd}$ of $a$ and $q$.
The symbols $\varphi,\mu,\omega,$ and $\zeta$ denote the Euler phi, M\"obius, prime omega, and Riemann zeta functions, respectively. For $n\in\N$, $\sum_{d|n}$ denotes a sum over positive divisors $d$ of $n$. For $x\in\R,$ we write $\lfloor x\rfloor$ for the greatest integer less than or equal to $x$ and $\{x\}=x-\lfloor x\rfloor$ for the fractional part of $x$.

\section{Proof of Theorem \ref{thm.Main0}}

First note that it is sufficient to prove the theorem for $X\in\N$, since interpolating between integers in the main term adds an error of at most $O(\sigma_{-1}(n)X)$. For $n\in\Z$ and $X\in\N$, write $S(n,X)=S_{2,\|\cdot\|_\infty}(n,X)$,
\begin{align*}
	S_{leq}(n,X)&=\left\{(a,b,q,r)\in\Z^4 : |a|,|b|,|q|\le X, |r|\le |q|, ar-bq=n\right\},~\text{and}\\
	S_{eq}(n,X)&=\left\{(a,b,q,r)\in\Z^4 : |a|,|b|,|q|\le X, |r|= |q|, ar-bq=n\right\}.
\end{align*}
Then $S(n,X)=S(-n,X), S_{leq}(n,X)=S_{leq}(-n,X)$, and
\begin{align}
	S(n,X)&=S_{leq}(n,X)+S_{leq}(-n,X)-S_{eq}(n,X)\nonumber\\
	&=2S_{leq}(n,X)-S_{eq}(n,X).\label{eqn.S_dSum-1}
\end{align}
For $n\in\N$ we have that
\begin{align*}
	S_{leq}(n,X)=\sum_{d|n}S_d,
\end{align*}
where
\begin{align*}
	S_d=\sum_{\substack{|q|\le X\\(q,n)=d}}\sum_{\substack{|a|,|b|\le X\\|r|\le |q|\\ar-bq=n}}1.
\end{align*}
For the remainder of the proof we will simplify notation by writing $q_d=q/d, n_d=n/d, X_d=X/d,$ and so on. First of all we have that
\begin{align}\label{eqn.S_dSum0}
	S_d&=\sum_{\substack{|q_d|\le X_d\\(q_d,n_d)=1}}\sum_{\substack{|a|,|b|\le X\\|r|\le d|q_d|\\ar-bdq_d=n}}1.
\end{align}
Note that, in order to have $ar-bq=n$, it is necessary that $(a,q)|n$, which implies that $(a,q)|d$. Therefore there is no loss in restricting the sum over $a$ in the previous equation to divisors of $d$. Furthermore, if $(a,q)=e$, for some $e|d$, then it must be the case that $\frac{d}{e}|r$.  This means that \eqref{eqn.S_dSum0} is equal to
\begin{align*}
\sum_{\substack{|q_d|\le X_d\\(q_d,n_d)=1}}\sum_{e|d}\sum_{\substack{|a|\le X\\(a,q)=e}}\sum_{\substack{|r|\le|q|\\\frac{d}{e}|r}}\sum_{\substack{|b|\le X\\a_er_{d/e}-bq_d=n_d}}1.
\end{align*}
The condition that $(a,q)=e$ guarantees that $(a_e,q_d)=1$. Once $q, a,$ and $r$ are chosen in the sums above, the equation $a_er_{d/e}-bq_d=n_d$ will have a solution $|b|\le X$ (which obviously then must be unique) if and only if $r_{d/e}=n_da_e^{-1}~\mathrm{mod}~q_d$ and
\begin{equation}\label{eqn.S_dSum2}
\left|\frac{a_er_{d/e}-n_d}{q_d}\right|\le X.
\end{equation}
Therefore 
\begin{align*}
	S_d=\sum_{\substack{|q_d|\le X_d\\(q_d,n_d)=1}}\sum_{e|d}\sum_{\substack{|a_e|\le X_e\\(a_e,q_e)=1}}~\sum_{\substack{|r_{d/e}|\le|q_{d/e}|\\r_{d/e}=n_da_e^{-1}~\mathrm{mod}~q_d\\\text{\eqref{eqn.S_dSum2} holds}}}1.
\end{align*}
Now let us consider the effect of removing condition \eqref{eqn.S_dSum2} from the inner sum. First, note that the conditions $|a_e|\le X_e$ and $|r_{d/e}|\le |q_{d/e}|$ already guarantee that 
\begin{equation*}
	\frac{a_er_{d/e}-n_d}{|q_d|}\le X.
\end{equation*}
For the other inequality, suppose that $q_d$ and $e$ are chosen, that $r_{d/e}=n_da_e^{-1}~\mathrm{mod}~q_d$, and that
\begin{equation*}
	\frac{a_er_{d/e}-n_d}{|q_d|}< -X.
\end{equation*}
Then we have that $a_er_{d/e}=n_d~\mathrm{mod}~q_d$ and that
\begin{equation*}
	-|q_d|X\le a_er_{d/e}< -|q_d|X+n_d.
\end{equation*}
Using the assumption that $X\in\N$, we have that the number of choices for the integer $c=a_er_{d/e}$ is equal to $\left\lfloor\frac{n_d}{|q_d|}\right\rfloor$. For each such integer $c$ the number of different values of $a_e$ and $r_{d/e}$ with $c=a_er_{d/e}$ is at most
\[\sigma_0(c)\ll_\epsilon \left(|q_d|X\right)^\epsilon.\]
Therefore
\begin{align*}
	S_d=M_d+E_d,
\end{align*}
with
\begin{align}
	M_d&=\sum_{\substack{|q_d|\le X_d\\(q_d,n_d)=1}}\sum_{e|d}\sum_{\substack{|a_e|\le X_e\\(a_e,q_e)=1}}~\sum_{\substack{|r_{d/e}|\le|q_{d/e}|\\r_{d/e}=n_da_e^{-1}~\mathrm{mod}~q_d}}1\nonumber\\
	&=\sum_{\substack{|q_d|\le X_d\\(q_d,n_d)=1}}\sum_{e|d}2e\left(\frac{2X_e\varphi(q_e)}{|q_e|}+O\left(2^{\omega(q_e)}\right)\right)\nonumber\\
	&=4X\sum_{e|d}\sum_{\substack{|q|\le X\\(q,n)=d}}\frac{\varphi(q/e)}{|q/e|}+O\left(\sum_{e|d}e\sum_{\substack{|q|\le X\\(q,n)=d}}2^{\omega(q/e)}\right)\label{eqn.S_dSum3}
\end{align}
and
\begin{align*}
	E_d&\ll_\epsilon \sum_{\substack{|q_d|\le X_d\\(q_d,n_d)=1}}\sum_{e|d}\frac{n_d}{|q_d|}\left(|q_d|X\right)^\epsilon\ll_\epsilon \frac{\sigma_0(d)}{d^{1+\epsilon}}\left(nX^{2\epsilon}\right).
\end{align*}
Now we consider the sum of these quantities over divisors of $n$. For the main term in \eqref{eqn.S_dSum3} we have that
\begin{align*}
	\sum_{d|n}\left(4X\sum_{e|d}\sum_{\substack{|q|\le X\\(q,n)=d}}\frac{\varphi(q/e)}{|q/e|}\right)&=4X\sum_{e|n}\sum_{\substack{|q|\le X\\e|q}}\frac{\varphi(q/e)}{|q/e|}\\
	&=4X\sum_{e|n}\sum_{|q_e|\le X_e}\frac{\varphi(q_e)}{|q_e|}\\
	&=\frac{48\sigma_{-1}(n)}{\pi^2}X^2+O\left(\sigma_0(n)X\log X\right).
\end{align*}
For the error term in \eqref{eqn.S_dSum3} we have that
\begin{align*}
	\sum_{d|n}\sum_{e|d}e\sum_{\substack{|q|\le X\\(q,n)=d}}2^{\omega(q/e)}=\sum_{e|n}e\sum_{|q_e|\le X_e}2^{\omega(q_e)}\ll \sigma_0(n)X\log X.
\end{align*}
For the sum of the $E_d$ error terms, adjusting the choice of $\epsilon$ above and the corresponding constant, we obtain
\begin{align*}
	\sum_{d|n}E_d\ll_\epsilon nX^\epsilon.
\end{align*}
Finally, it is easy to show that $S_{eq}(n,X)\ll \sigma_0(n)X$. Combining these estimates in formula  \eqref{eqn.S_dSum-1} completes the proof of the theorem.

\section{Proof of Theorem \ref{thm.Main1}}
First consider the following basic lemma.
\begin{lemma}\label{lem.p-adic1}
Let $p$ be prime and suppose that $n=p$ or that $n=p^2$. For every integer $c$ satisfying $0\le c< n$, there exist integers $a$ and $q$ satisfying $|a|\le n^{1/2}$, $0<q< n^{1/2}$, and
\[aq^{-1}=c ~\mathrm{mod}~ n.\]
\end{lemma}
\begin{proof}[Proof of lemma]
Let
\[\mc{A}=\left\{(b,r): 0\le b\le n^{1/2}, 0\le r<n^{1/2}\right\}.\]
It is easy to check that, whether $n=p$ or $p^2$,
\[|\mc{A}|\ge n+1.\]
Therefore, by the Dirichlet pigeonhole principle, there are distinct $(a_1,q_1),(a_2,q_2)\in\mc{A}$ with
\[a_1-q_1c=a_2-q_2c~\mathrm{mod}~n.\]
Take $a=a_1-a_2$ and $q=q_1-q_2$ and assume without loss of generality, by switching the signs if necessary, that $q\ge 0$. If $q=0~\mathrm{mod}~p$ then, since $q<n^{1/2}\le p$, it must be the case that $q=0$. Since $|a|\le n^{1/2}$, this forces $a=0$ as well, which is a contradiction. Therefore $q>0$, and it is invertible modulo $n$, which completes the proof of the lemma.
\end{proof}

Let us continue to write $n=p$ or $p^2$, depending on which of the two cases in the statement of Theorem \ref{thm.Main1} is being considered. Define $\Omega\in\Z^2$ by
\[\Omega=\left\{(a,q)\in\Z^2: |a|\le n^{1/2}, 0<q<n^{1/2}, (a,q)=1\right\},\]
and define $T:\Omega\rar(\Z/n\Z)$ by
\[T(a,q)=aq^{-1}~\mathrm{mod}~n.\]
It is easy to see from Lemma \ref{lem.p-adic1} that the map $T$ is surjective. Distinct points $(a,q),(b,r)\in\Omega$ will have the same image under $T$ if and only if $ar-bq=\pm n$. It follows from this that every point in $\Z/n\Z$ is the image of either one or two points in the domain, and that
\begin{align*}
	n=|\Omega|-\#\{((a,q),(b,r))\in\Omega\times\Omega : ar-bq=n\}.
\end{align*}
Finally, by considering the maps $(a,q)\leftrightarrow (-a,-q)$ and $(b,r)\leftrightarrow (-b,-r)$ (which require also mapping $n$ to $-n$, but do not change the quantities we are counting), and by subtracting off contributions along the boundaries of the regions, we have that
\begin{align*}
S_{2,\|\cdot\|_\infty}(n,\sqrt{n})&=4\#\{((a,q),(b,r))\in\Omega\times\Omega : ar-bq=n\}+O(n^{1/2})\\
&=4\left(|\Omega|-n\right)+O(n^{1/2})\\
&=4\left(\frac{12}{\pi^2}-1\right)n+O(n^{1/2}\log n).
\end{align*}
This completes the proof of the theorem.

\vspace{.15in}
		
{\footnotesize
\noindent
Department of Mathematics\\
University of Houston\\
Houston, TX, United States\\
kkavita@cougarnet.uh.edu\\
haynes@math.uh.edu\\
szprasal@cougarnet.uh.edu
			
}

	\end{document}